\documentclass[10pt,twocolumn,twoside]{IEEEtran}

\usepackage{todonotes}
\usepackage{generic}
\usepackage{cite}
\usepackage{amsmath,amssymb,amsfonts,amsthm}
\usepackage{algorithmic}
\usepackage{graphicx}
\usepackage{textcomp}

\usepackage[english]{babel}
\usepackage{blindtext}
\usepackage{amsfonts,amssymb}
\usepackage{caption}
\usepackage{subcaption}
\usepackage{enumitem}
\usepackage{float}
\usepackage{bm}
\usepackage{newtxmath}
\usepackage{mathrsfs}
\usepackage{longtable}
\usepackage{pdfpages}
\usepackage{verbatim}
\usepackage{tikz}
\tikzstyle{block} = [draw, fill=white, rectangle, 
minimum height=3em]
\tikzstyle{sum} = [draw, fill=white, circle, node distance=1cm]
\tikzstyle{input} = [coordinate]
\tikzstyle{output} = [coordinate]
\tikzstyle{pinstyle} = [pin edge={to-,thin,black}]
\usepackage{tikz-network}
\usetikzlibrary{positioning}
\usepackage{bbm}
\usetikzlibrary{graphs,graphs.standard,calc}
\graphicspath{ {./images/} }
\usepackage{todonotes}
\usepackage{amsthm}
\newtheorem{definition}{Definition}
\newtheorem{theorem}{Theorem}
\newtheorem{Lemma}{Lemma}
\newtheorem{asp}{Assumption}
\newtheorem{prob}{Problem}
\newtheorem{proposition}{Proposition}
\newtheorem{cor}{Corollary}

\newtheorem{remark}{Remark}

\DeclareMathOperator{\Null}{\mathrm{Ker}}
\DeclareMathOperator{\Span}{\mathrm{span}}

\DeclareMathOperator{\diag}{\mathrm{diag}}
\DeclareMathOperator{\range}{\mathrm{IM}}

\begin{document}
\title{Extending the Leader-First Follower Structure for Bearing-only Formation Control on Directed Graphs}
\author{Jiacheng Shi and Daniel Zelazo~\IEEEmembership{Senior~Member,~IEEE}  \thanks{This work was supported by the Israel Science Foundation grant no. 453/24 and the Technion Autonomous Systems Program. 
 (Corresponding author: Daniel Zelazo) J. Shi (e-mail: sjc199921@gmail.com) is with the the Technion Autonomous Systems Program and D. Zelazo (e-mail: dzelazo@technion.ac.il) is with the Stephen B. Klein Faculty of Aerospace Engineering, both at the Technion-Israel Institute of Technology, Haifa 3200003, Israel. 
}
}
\maketitle

\begin{abstract}
\textcolor{black}{This work introduces a generalization of the leader-first follower (LFF) graph structure for solving the bearing-only formation control problem on directed graphs.
} The first  contribution provides an equilibrium, stability, and convergence analysis for a one-follower, multi-leader system (which is not an LFF graph).  We then propose an extension to the LFF structure, termed \emph{ordered} LFF graphs, that allows for additional forward directed edges to be included.  Using the results of the one-follower multi-leader system we show that the ordered LFF graphs can be used to solve the directed bearing-only formation control problem.  We also show that these structures offer improved convergence speed as compared to the LFF graphs.  Numerical simulations are provided to validate the results.

\end{abstract}


\section{INTRODUCTION}

Formation control has obtained significant attention across a wide range of fields, including robotics \cite{lewis1997high}, aerial and ground vehicle networks \cite{ren2006consensus}, and swarm robotics \cite{xu2014behavior}.  The primary task in formation control is to drive a team of autonomous systems into a desired spatial configuration.  As a cornerstone problem in multi-agent coordination, one of the challenges in formation control is the design of distributed control protocols that balance the sparsity of information exchange with the performance of the system.  In this direction, there is a considerable body of literature that addresses this problem for a variety of different formation and sensing constraints.  These include position-constrained formation control \cite{ren2007distributed}, displacement-constrained formation control \cite{de2020maneuvering}, distance-constrained formation control \cite{Krick2009}, bearing-constrained formation control \cite{zhao2015bearing}, and most recently angle-constrained formation control \cite{Chen_TAC2021}. The reader is referred to \cite{oh2011formation, ahn2019formation, Zhao_CSM2019} for an overview of this area.

Despite recent progress in the study of formation control problems, there remains a large gap between the theoretical advances and their real-world implementation.  Indeed, many works on multi-agent systems assume undirected communication and sensing networks.  In reality, sensing employed in, for example, robotic systems is inherently uni-directional.  In other words, if agent $i$ can sense agent $j$, it is not necessarily true that agent $j$ can sense agent $i$.  The problem of directed formation control was originally studied by Hendrickx et. al in \cite{Hendrickx2006} where the notion of \emph{persistence} was introduced to describe consistency in directed distance-constraint frameworks.  This work however did not consider formation control strategies for directed frameworks, but rather attempted to characterize the feasibility sets of directed frameworks.  Several works studied the stability and equilibria of very small or peculiar formations in distance-constrained frameworks; see \cite{Belabbas2013, Babazadeh2020,Yu_SIAMJCO2009}.  For bearing-constrained formation control problems similar approaches have been taken. In \cite{Zhao2015persistance}, the notion of bearing persistence was introduced, although a stability proof for the corresponding linear bearing-based directed formation control law remains open.  This work was extended in \cite{Sun2023} but focused on the rigidity-theoretic understanding of bearing persistence rather than the stability and convergence of directed formation control strategies.    For general directed constraint networks (both distance and bearings), it remains an open challenge to i) characterize the equilibria of formation dynamics, ii) assess the stability of the equilibria, and iii) determine graph and rigidity theoretic conditions for the existence of directed frameworks that admit solutions to the formation control problem.

To illustrate the challenge associated to formation control over directed graphs, consider the example in Figure \ref{fig:ex1}.  Here we task a team of integrator agents embedded in the plane to obtain a hexagonal formation, while inter-agent interaction is restricted according to the sensing graphs in Figure \ref{fig:ex1_tf_1} and \ref{fig:ex1_tf_2} respectively.  Note that the only difference is the direction of the sensing edge between agent $1$ and $4$. Figures \ref{fig:ex1_tr_1} and \ref{fig:ex1_tr_2} show the different agent trajectories, when implementing the formation control strategy proposed in \cite{zhao2015bearing} but adapted for directed sensing.\footnote{Details of this control law will be reviewed in Section \ref{sec.bearingformations}.} The graph in Figure \ref{fig:ex1_tf_1} is not able to converge to the correct formation while the one in Figure \ref{fig:ex1_tf_2} is. Of note is that the undirected version of both graphs are minimally infinitesimally bearing rigid, and therefore the undirected implementation of the control law is guaranteed to converge to the correct formation \cite{zhao2015bearing}.
\begin{figure*}[h]
    \centering
    \begin{subfigure}[b]{0.23\linewidth}
    \centering
    \begin{tikzpicture}
    \Vertex[size=.5,label=$p_1$,x=0.5,y=0.866]{v1}
    \Vertex[size=.5,label=${p_2}$,x=-0.5,y=0.866]{v2}
    \Vertex[size=.5,label=$p_3$,x=-1.2,y=0]{v3}
    \Vertex[size=.5,label=$p_4$,x=-0.5,y=-0.866]{v4}
    \Vertex[size=.5,label=$p_5$,x=0.5,y=-0.866]{v5}
    \Vertex[size=.5,label=$p_6$,x=1.2,y=0]{v6}
    \Edge[Direct](v2)(v1)
    \Edge[Direct](v3)(v1)
    \Edge[Direct](v3)(v2)
    \Edge[Direct](v4)(v2)
    \Edge[Direct,color=red](v3)(v4)
    \Edge[Direct](v4)(v6)
    \Edge[Direct](v5)(v3)
    \Edge[Direct](v5)(v4)
    \Edge[Direct](v6)(v1)
    \Edge[Direct](v6)(v3)
    \Edge[Direct](v6)(v5)
    \end{tikzpicture}
    \caption{A bad sensing graph.}
    \label{fig:ex1_tf_1}
    \end{subfigure}
    \begin{subfigure}[b]{0.23\linewidth}
        \centering
    \includegraphics[width=1\linewidth]{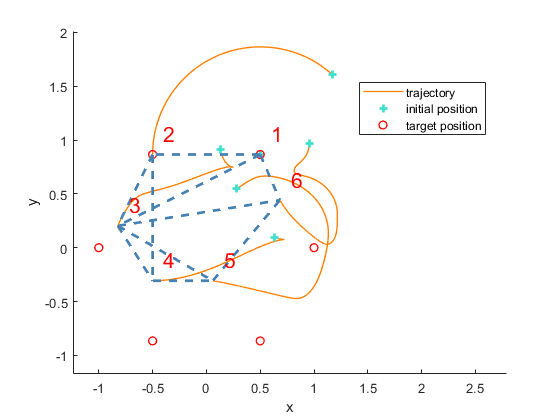}
    \caption{Agents fail to converge to the correct formation.}
    \label{fig:ex1_tr_1}
    \end{subfigure}
    \begin{subfigure}[b]{0.23\linewidth}
    \centering
    \begin{tikzpicture}
    \Vertex[size=.5,label=$p_1$,x=0.5,y=0.866]{v1}
    \Vertex[size=.5,label=${p_2}$,x=-0.5,y=0.866]{v2}
    \Vertex[size=.5,label=$p_3$,x=-1.2,y=0]{v3}
    \Vertex[size=.5,label=$p_4$,x=-0.5,y=-0.866]{v4}
    \Vertex[size=.5,label=$p_5$,x=0.5,y=-0.866]{v5}
    \Vertex[size=.5,label=$p_6$,x=1.2,y=0]{v6}
    \Edge[Direct](v2)(v1)
    \Edge[Direct](v3)(v1)
    \Edge[Direct](v3)(v2)
    \Edge[Direct](v4)(v2)
    \Edge[Direct,color=red](v4)(v3)
    \Edge[Direct](v4)(v6)
    \Edge[Direct](v5)(v3)
    \Edge[Direct](v5)(v4)
    \Edge[Direct](v6)(v1)
    \Edge[Direct](v6)(v3)
    \Edge[Direct](v6)(v5)
    \end{tikzpicture}
    \caption{A good sensing graph.}
    \label{fig:ex1_tf_2}
    \end{subfigure}
    \begin{subfigure}[b]{0.23\linewidth}
        \centering
    \includegraphics[width=1\linewidth]{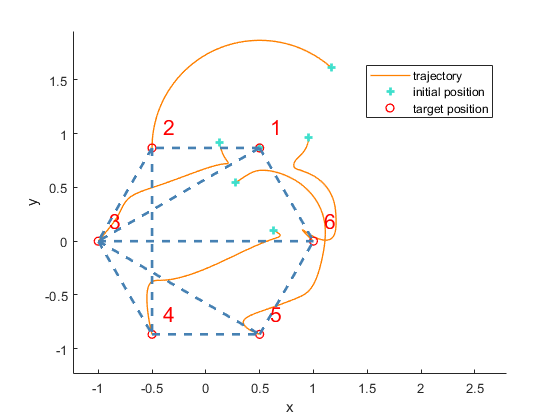}
    \caption{Agents do converge to the correct formation.}
    \label{fig:ex1_tr_2}
    \end{subfigure}
    \caption{Example demonstrating the challenge of formation control with directed sensing.}
    \label{fig:ex1}
\end{figure*}

\textcolor{black}{An important step toward understanding bearing-only formation control under directed sensing was made in \cite{trinh2018bearing}.}
In this work, they showed that for a special class of directed graphs, known as the \emph{leader first follower} (LFF) graphs, the bearing-only formation control law proposed in \cite{zhao2015bearing} but adapted for directed sensing almost globally converges to the desired formation.  \textcolor{black}{{LFF} graphs have a special structure where one node has no out edges (the leader), the first follower node has only one out edge towards the leader, and all other nodes have exactly two outgoing edges.} The main idea of this work is that these LFF graphs lead to a cascade system structure facilitating the stability and convergence proof.  Nevertheless, LFF graphs are restrictive and this work aims to extend the class of directed graphs that can be used to solve the bearing formation control problem with directed sensing.


The contribution of this paper focuses on extending the LFF structure to solve the bearing-only formation control problem over directed graphs.  In this direction, our first result considers a simpler setup consisting of a single follower agent and many leaders.  We provide an analysis to characterize the equilibria configuration for this system and discuss their stability and convergence properties.  It turns out that the analysis of this simpler system is crucial for a more general extension of LFF graphs, which leads to our second contribution.  We extend the BOFC framework to accommodate augmented LFF graphs, incorporating additional forward-directed edges while preserving the cascade structure of the system.  \textcolor{black}{Compared to LFF graphs, the augmented LFF graphs allow for nodes to have \emph{at least} 2 outgoing edges as opposed to exactly 2.}  We provide an analysis of the corresponding equilbria and stability of the system. Lastly, we provide simulation studies to demonstrate the feasibility and performance improvements enabled by the proposed extensions, showcasing faster convergence rates and increased flexibility in network design. These results offer a significant step toward more robust and adaptable BOFC solutions in directed sensing scenarios.

The remainder of this paper is organized as follows. Section \ref{sec.bearingformations} provides a brief overview of bearing-only formation control (BOFC) and revisits key results for undirected and directed graphs. Section \ref{sec.extendLFF} presents our main contributions, beginning with an analysis of the one-to-many BOFC setup and extending to LFF graph structures with additional forawrd edges. Section \ref{sec.simulation} illustrates the theoretical results with numerical simulations, highlighting performance improvements, and demonstrating the feasibility of the proposed methods. Finally, Section \ref{sec.conclusion} concludes the paper and discusses potential directions for future work.

\paragraph*{Notations} 
Throughout this paper, $\mathbb{R}^n$ denotes the $n$-dimensional real vector space, and $\|\cdot\|$ represents the Euclidean norm for vectors. The identity matrix of size $n \times n$ is denoted by ${I}_n$, while $\vmathbb{1}_n$ represents the all-ones vector of dimension $n$. For a set of matrices or vectors $\{M_i\}_{i=1}^n$, $\mathrm{diag}(M_i)_{i=1}^n$ denotes the block diagonal matrix with $M_i$ as its diagonal blocks.  When the set of matrices are clear from context we write only $\mathrm{diag}(M_i)$.  The Kronecker product is denoted by $\otimes$. The image and kernel of a matrix $M$ are represented by $\range(M)$ and $\Null(M)$, respectively. \textcolor{black}{A symmetric matrix $M$ is positive definite (semi-definite) if the quadratic form satisfies $x^TMx > 0$ ($x^TMx\geq0$) and is denoted as $M \succ 0$ ($M\succeq 0$). The unit sphere in $\mathbb{R}^{\mathrm d}$ is the set $\mathbb{S}^{\mathrm d-1}=\{ x \in \mathbb{R}^{\mathrm d} \,:\, \|x\|=1\}.$}

\section{Bearing-Only Formation Control}\label{sec.bearingformations}

In this section we provide a brief overview of bearing only formation control (BOFC) problem for the both undirected and directed settings.  We begin with the general setup and then present the key results from \cite{zhao2015bearing} and \cite{trinh2018bearing}.

We consider a network of $n$ agents described by the integrator dynamics,
\begin{align}\label{integrator}
    \dot p_i(t) & = u_i(t), \; i=1,\ldots, n,
\end{align}
where $p_i(t),u_i(t)\in \mathbb{R}^\mathrm{d}$ are, respectively, the position and velocity control of each agent.  Here, $\mathrm{d}$ represents the ambient dimension for the system, and we typically assume $\mathrm{d}\in \{2,3\}$.  The vector $p(t) = \begin{bmatrix} p_1(t)^T & \cdots & p_n(t)^T\end{bmatrix}^T$ is denoted as the system \emph{configuration}.

Agents can interact with each other according to a static graph, described by the pair $\mathcal G=(\mathcal V,\mathcal E)$.  Here $\mathcal V =\{1,\ldots,n\}$ is the \emph{node set}, and $\mathcal E \subseteq \mathcal V \times \mathcal V$ is the \emph{edge set}.  The notation $ij\in\mathcal E$ denotes that node $i \in \mathcal V$ is connected to node $j\in \mathcal V$.  The graph may be \emph{undirected}, in which case if $ij \in \mathcal E$, then $ji \in \mathcal{E}$, or \emph{directed} which means that $ij \in \mathcal E$ does not imply that $ji\in \mathcal E$ \cite{Mesbahi2010}.

A \emph{bearing formation} is the vector $\bm{\mathrm g} = \begin{bmatrix} \bm{\mathrm g}_{1}^T & \cdots & \bm{\mathrm g}_{|\mathcal E|}\end{bmatrix} \in \mathbb{R}^{\mathrm d|\mathcal E|}$, \textcolor{black}{with $\bm{\mathrm g}_{i} \in \mathbb{S}^{\mathrm d-1}$},  specifying the desired bearing between neighboring agents. Naturally, we are concerned with bearing formations that are actually realizable by some configuration $\mathrm p \in \mathbb{R}^{\mathrm{d}n}$.  In this direction, we introduce the notion of the \emph{bearing function}, $F_B:\mathbb{R}^{\mathrm{d}n}\to\mathbb{R}^{\mathrm{d}|\mathcal E|}$, defined as 
\begin{align}
    F_B(p)&=\begin{bmatrix} g_1^T & \cdots & g_{|\mathcal E|}^T\end{bmatrix}^T,
\end{align}
where for edge $k=ij\in\mathcal{E}$, the vector $g_k$ is the unit vector pointing from $p_i$ to $p_j$,
\begin{align}\label{bearing_vector}
g_k := g_{ij} = \frac{p_j-p_i}{\|p_j-p_i\|} \textcolor{black}{\in \mathbb{S}^{\mathrm d -1}}.
\end{align}
\textcolor{black}{We also assume that all agents have access to a common global reference frame and each agent can sense this relative bearing vector to its neighbor agents.}

With this notation, we define a \emph{bearing formation}, $(\mathcal{G},g)$,  by associating the edges in the graph $\mathcal{G}$ with the bearing measurements $g=F_B(p)$.  We now provide a formal definition for a realizable bearing formation.
\begin{definition}
    A bearing formation $(\mathcal G,\mathrm g)$ is \emph{realizable} in $\mathbb{R}^{\mathrm{d}}$ if there exists a configuration $\mathrm p\in\mathbb{R}^{\mathrm{d}n}$ satisfying $\mathrm p \in F_B^{-1}(\mathrm g)$.
\end{definition}

We now present the general bearing-only formation control problem.  Note that the above set-up and the following problem statement does not depend on whether $\mathcal G$ is directed or undirected.  In the sequel we will explore how directedness affects the solutions.

\begin{prob}\label{prob.formationcontrol}
Consider a collection of $n$ agents described by \eqref{integrator} that interact over a graph $\mathcal G$ and let $(\mathcal G, \bm{\mathrm g})$ be a realizable bearing formation.  Design a distributed control for each agent using only bearing measurements obtained from neighboring agents, i.e., a control of the form $u_i(t) = \sum_{ij\in \mathcal E}\kappa_{ij}(g_{ij},\bm{\mathrm g}_{ij})$ that
 drives the system to the target formation, i.e.,
    $$ \lim_{t\to \infty} g(t)=\lim_{t\to \infty}F_B(p(t)) = \bm{\mathrm{g}}.$$ 
\end{prob}
\textcolor{black}{The functions \( \kappa_{ij} : \mathbb{S}^{\mathrm d-1} \times \mathbb{S}^{\mathrm d-1} \to \mathbb{R}^{\mathrm d} \) can be interpreted as the control implemented on each edge in the graph.} We now review the solutions to Problem \ref{prob.formationcontrol} for the undirected and directed cases.

\subsection{BOFC for Undirected Graphs}\label{subsec_undirBOFC}

A solution to Problem \ref{prob.formationcontrol} for undirected graphs was initially proposed in \cite{zhao2015bearing}.  The control strategy has the form
\begin{align}\label{undirBOFC}
u_i(t) & = -\sum_{ij\in \mathcal E}P_{g_{ij}(t)} \bm{\mathrm{g}}_{ij}, \, i=1,\ldots,n,
\end{align}
where $P_{x} \in \mathbb{R}^{\mathrm d \times \mathrm d}$ is the orthogonal projection matrix \textcolor{black}{onto the subspace orthogonal to the vector $x$}, defined as 
\begin{align}\label{projmat}
    P_{x} = I_{\mathrm d} - \frac{x}{\|x\|}\frac{x^T}{\|x\|}.
\end{align}
It is convenient to represent the control \eqref{undirBOFC} in an aggregated matrix form as
\begin{align}\label{undirBOFC_mat}
    u(t)=\bar{H}^T \diag (P_{g_{ij}(t)}) \bm{\mathrm{g},}
\end{align}
where \textcolor{black}{$\bar H = H \otimes I_{\mathrm d} \in \mathbb{R}^{{\mathrm d}|\mathcal{E}| \times {\mathrm d}n}$, with $H \in \mathbb{R}^{|\mathcal{E}| \times n}$  the incidence matrix associated with the graph $\mathcal G$, defined as
\begin{equation}
    [ H]_{ki}=
    \begin{cases}
    1,&  \text{node } i \text{ is  positive   end   of   edge }  e_k\\
    -1,&  \text{node } i \text{ is  negative   end   of   edge }  e_k\\
    0,& \text{otherwise}    .\end{cases}
\end{equation}}
The matrix form of the control \eqref{undirBOFC_mat} turns out to be related to the \emph{bearing rigidity matrix} for a bearing framework.  The bearing rigidity matrix is defined by the Jacobian of the bearing function $F_B$, and has the form
\begin{align*}
    R_B(p) =  \diag \left(\frac{P_{g_{ij}}}{d_{ij}}\right)\bar H,
\end{align*}
where $d_k = d_{ij} = \|p_i-p_j\|$ is the distance between points $p_i$ and $p_j$ when $ij\in \mathcal{E}$.  With this definition, the control can be expressed as
$$u(t) = \diag(d_{ij})R_B(p)^T\bm{\mathrm{g}}.$$
For details on bearing rigidity theory and the bearing rigidity matrix, the reader is referred to \cite{zhao2015bearing}.  
The main result from \cite{zhao2015bearing} states that if the target bearing formation is infinitesimally bearing rigid, then the control \eqref{undirBOFC_mat} almost globally and exponentially converges to the desired formation.  

\subsection{BOFC for Directed Graphs}\label{subsec_dirBOFC}

A natural approach for solving Problem \ref{prob.formationcontrol} for directed graphs is to simply try the same control as in \eqref{undirBOFC}, where the sum over the edges $ij\in \mathcal{E}$ are now directed.  The aggregated matrix version of the control takes a slightly modified form arising from a redefinition of the incidence matrix for directed graphs. \textcolor{black}{Define now the \emph{out-incidence} matrix for the directed graph $\mathcal G$, denoted as $ H_{\otimes}\in \mathbb{R}^{ |\mathcal E| \times  n}$, as
\begin{equation}
    [ H_{\otimes}]_{ki}=
    \begin{cases}
    1,&  \text{node } i \text{ is  positive   end   of   edge }  e_k\\
    0,& \text{otherwise}   \end{cases},
\end{equation}
and $\bar H_{\otimes} = \bar H \otimes I_{\mathrm d}$.}
Then the proposed control takes the form
\begin{equation}
    u(t)=\bar{H}_{\otimes}^T \diag (P_{g_{ij}(t)}) \bm{\mathrm{g}}.
    \label{eqn:law_di}
\end{equation}

We now recall the example shown in Figure \ref{fig:ex1}. In this example, the undirected version of the graph leads to a target bearing formation that is minimally infinitesimally bearing rigid, and therefore the undirected control law \eqref{undirBOFC_mat} solves Problem \ref{prob.formationcontrol}.  However, the example shows that depending on what orientation is used to solve the problem over directed graphs, it may or may not converge to the correct formation. One of the main contributions of \cite{trinh2018bearing} was to propose a class of directed graphs that solve Problem \ref{prob.formationcontrol} using the control \eqref{eqn:law_di}.

\begin{definition} A directed graph is a \emph{leader-first follower } (LFF) graph if 
\begin{itemize}
    \item[$i$)] there is a vertex with no outgoing edges, denoted as the \emph{leader}, assigned the label $v_1$; 
    \item[$ii$)] there is a vertex with only one outgoing edge pointing to the leader, denoted as the \emph{first follower} assigned the label $v_2$;
    \item[$iii$)] every vertex other than the leader and first follower has exactly two outgoing edges;
    \item[$iv$)] for every directed edge $e_{ij}$, the label is ordered as $i>j$.
\end{itemize}
\label{def:HCLFF}
\end{definition}
An example of an LFF graph is shown in Figure \ref{fig.LFF}. Here, the leader is identified by the red node and the first follower by the dark blue node.  We denote any edge $e_{ij}$ with $i > j$ as a \emph{forward edge}.  With this notion we see LFF graphs consist only of forward edges.
\begin{figure}[!h]
        \centering
        \begin{tikzpicture}
    \Vertex[size=.5,label=$v_1$,color=red,x=0,y=0]{v1}
    \Vertex[size=.5,label=$\color{white}{v_2}$,color=blue,x=1,y=0]{v2}
    \Vertex[size=.5,label=$v_3$,x=2,y=0]{v3}
    \Vertex[size=.5,label=$v_4$,x=3,y=0]{v4}
    \Vertex[size=.5,label=$v_{n}$,x=5,y=0]{vi}
    \node at (4,0) {$\cdots$};
    \Edge[Direct](v2)(v1)
    \Edge[Direct,bend=30](v3)(v1)
    \Edge[Direct,bend=-30](v3)(v2)
    \Edge[Direct,bend=30](v4)(v2)
    \Edge[Direct,bend=-30](v4)(v3)
    \Edge[Direct,bend=30](vi)(v3)
    \Edge[Direct,bend=-30](vi)(v1)
\end{tikzpicture}
       
        \caption{An example of an LFF graph.}\label{fig.LFF}
    \end{figure}
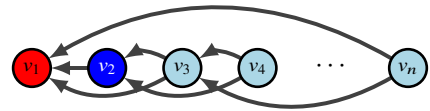

Of note for the BOFC on these LFF graphs is the resulting cascade structure of the closed-loop system.  Indeed, for LFF graphs, the dynamics become
\begin{align*}\label{dirBOFC_LFF}
\begin{cases}\dot p_1(t) &= 0 \\
\dot p_2(t) &= \textcolor{black}{-P_{g_{21}(t)}\bm{\mathrm g}_{21} }\\
\dot p_i(t) &= -\sum_{ij\in \mathcal E}P_{g_{ij}(t)} \bm{\mathrm{g}}_{ij}, \, i=3,\ldots,n
\end{cases}.
\end{align*}

The main result then from \cite{trinh2018bearing} states that if the target formation is an LFF graph, the control \eqref{eqn:law_di} solves Problem \ref{prob.formationcontrol}.  The proof relies on the stability theorem of cascade systems \textcolor{black}{\cite[Theorem 4.1]{Seibert1990}}
and the fact that the leader and first follower effectively fix the centroid and scale of the formation.

\section{Extending the LFF Structure}
\label{sec.extendLFF}

Extending the class of directed graphs that can solve Problem \ref{prob.formationcontrol} poses significant challenges.  The moment the LFF structure is broken, then the cascade system analysis from \cite{trinh2018bearing} may no longer be applied.  Nevertheless, there is an interest to find additional structures that can be used, enabling a network designer to have more flexibility to design systems with additional properties such as performance or robustness.  To emphasize this point, we refer again to the example of Figure \ref{fig:ex1_tf_2} which solves Problem \ref{prob.formationcontrol} but is not an LFF graph indicating that such structures exist. 

To begin, we focus on a simple system comprised of one follower and many leaders.  This is not an LFF graph, and the analysis of this simpler problem will provide the framework needed to extend the LFF graphs.

\textcolor{black}{\begin{remark}
    The bearing-only formation control strategies summarized in Section \ref{sec.bearingformations} implicitly assume no inter-agent collisions - otherwise the bearing vectors would be undefined.  Both \cite{zhao2015bearing} and \cite{trinh2018bearing} explore sufficient conditions that ensure no collisions will occur, or discuss the possibility to augment the control with additional collision-avoidance schemes. In this work, we do not pursue these directions further, as our focus is on the nominal convergence properties of the control law under ideal sensing conditions. 
\end{remark}}

\subsection{BOFC with 1 Follower and Many Leaders (1-to-many)}

We consider $n\geq 3$ agents modeled by the dynamics \eqref{integrator} with $n-1$ leaders and one follower. We denote the leaders by the first $n-1$ nodes, and thus the follower is node $n$.  The directed graph has edges only of the form $ni \in \mathcal E$ for $i=1,\ldots, n-1$.  An example of this structure is shown in Figure \ref{fig:sensing_1dof}.\footnote{Note that for $n=1$ the graph is trivial, and for $n=2$ the case is  studied in
\cite{trinh2018bearing},  so we do not consider them.}

\begin{figure}[!h]
    \centering
    \begin{tikzpicture}
    \Vertex[size=.6,label=$v_1$,x=-1.5,y=0]{v1}
    \Vertex[size=.6,label=$v_2$,x=-0.5,y=0]{v2}
    \Vertex[size=.6,label=$\cdots$,x=0.5,y=0]{vi}
    \Vertex[size=.6,label=$v_{n-1}$,x=1.5,y=0]{vn}
    \Vertex[size=.6,label=$v_n$,x=0,y=1]{vc}
    \Edge[Direct](vc)(v1)
    \Edge[Direct](vc)(v2)
    \Edge[Direct](vc)(vi)
    \Edge[Direct](vc)(vn)
    \end{tikzpicture}
    \caption{The directed graph for the 1-to-many setup.}
    \label{fig:sensing_1dof}
\end{figure}
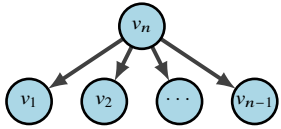

With this setup, and applying the control \eqref{undirBOFC}, the dynamics of each agent can be expressed as
\begin{equation}
\begin{aligned}
    \dot p_i(t)&=0,\, i=1,\ldots,n-1\\
    \dot p_n (t)&=-\sum_{nj\in \mathcal E}P_{g_{nj}(t)}\bm{\mathrm g}_{nj}.
    \label{sys:1DOF}
\end{aligned}
\end{equation}

We now introduce an assumption of the target formation for this setup.  
\begin{asp}
    The target bearing formation $(\mathcal{G},\bm{\mathrm{g}})$ is realizable.  Furthermore, there exists a configuration $\mathrm p \in F_B^{-1}(\bm{\mathrm g})$ such that $\mathrm p_i, \mathrm p_j$ and $\mathrm p_n$ are not collinear for any $i,j \in \{1,2,\ldots, n-1\}$.
    \label{asp:1dof_fc}
\end{asp}

Moreover, according to \eqref{sys:1DOF}, the leader nodes do not move.  Thus, the realizability question can be framed in terms of the initial conditions of the leader agents.  That is, for a configuration $\mathrm p$ satisfying Assumption \ref{asp:1dof_fc}, we must have that $p_i(0) = \mathrm p_i$ for $i=1,\ldots,n-1$. We are now prepared to present the first result which characterizes the equilibrium configuration of \eqref{sys:1DOF}.

To begin, we note that the solution to the equilibrium condition
\textcolor{black}{$$0 =  \sum_{nj\in \mathcal E}\kappa_{nj}(g_{nj}(t),\bm{\mathrm g}_{nj})=-\sum_{nj\in \mathcal E}P_{g_{nj}(t)}\bm{\mathrm g}_{nj}$$} can be expressed more naturally in terms of the bearings $\mathrm g_{nj}$.  Let the set $\mathcal X(\mathcal G,\bm{\mathrm g})$ be the set of bearings satisfying the equilibrium condition.  This set has a special structure, so we express it as the intersection of three sets,   
\begin{align}\label{Xeq_set}
\mathcal{X}(\mathcal{G},\bm{\mathrm{g}})&=\mathcal{X}_{eq}(\mathcal{G},\bm{\mathrm{g}})\cap \mathcal{C}_{1}\cap \mathcal{C}_{F_B}(\mathcal{G}).
\end{align}
The first set, $\mathcal{X}_{eq}(\mathcal{G},\bm{\mathrm{g}})$, characterizes all vectors that satisfy the equilibrium condition, i.e.,
\begin{align}\label{X_eq}
\mathcal{X}(\mathcal{G},\bm{\mathrm{g}}) = \left\{ x\in \mathbb{R}^{\mathrm d|\mathcal{E}|}: \textcolor{black}{\sum_{nj\in \mathcal E}\tilde\kappa_{ij}(x_{nj},\bm{\mathrm g}_{nj})=0}\right\},
\end{align}
\textcolor{black}{where $\tilde \kappa_{ij} : \mathbb{R}^{\mathrm d} \times \mathbb{S}^{\mathrm d -1} \to \mathbb{R}^{\mathrm d}$ is an extension of the domain of $\kappa_{ij}$ to allow non-unit vectors. The vector $x \in \mathbb{R}^{\mathrm d|\mathcal E|}$ is partitioned into $\mathrm d$-dimension blocks indexed by each edge in $\mathcal G$.} 
 Note that this set may include solutions that are \emph{not} bearing vectors (i.e., not unit-norm vectors).  The set 
\begin{align}\label{X_g}
   {\scriptsize{ \mathcal{C}_{1} = \{g={\tiny{\begin{bmatrix} g_1^T  \cdots  g^T_{|\mathcal E|}\end{bmatrix}^T}} \in \mathbb{R}^{\mathrm d |\mathcal E|} :\textcolor{black}{g_k \in \mathbb{S}^{\mathrm d-1}}, \, k=1,\ldots,|\mathcal E|\}\}}}
\end{align}
ensures all $\mathrm d$-vector entries of $\mathrm g$ are unit-norm vectors. Finally, the set 
\begin{align}\label{X_F}
   \mathcal{C}_{F_B}(\mathcal{G}) = \{g \in \mathbb{R}^{\mathrm d |\mathcal E|} \,:\, \exists \, p \in \mathbb{R}^{\mathrm d |\mathcal V|} \text{ s.t. } g = F_B(p)\},
\end{align}
considers only bearings that are realizable. In fact, it is true that $\mathcal C_{F_B(\mathcal G)} \subset \mathcal C_1$, but our analysis is aided by examining these sets separately.

Of interest now is to understand what vectors lie in the set $\mathcal X(\mathcal G,\bm{\mathrm g})$. We note that the equilibrium condition is nonlinear in the bearing $\mathrm g$, but is \emph{linear} in the target bearing $\bm{\mathrm g}$. In this direction, we may ask given a set of measured bearings $\mathrm g$, what are all possible target formations that result in an equilibrium?  This can also be characterized by the intersection of three sets,
\begin{align}\label{Yeq_set}
\mathcal{Y}(\mathcal{G},g)&=\mathcal{Y}_{eq}(\mathcal{G},g)\cap \mathcal{C}_{1}\cap \mathcal{C}_{F_B}(\mathcal{G}),
\end{align}
where
\begin{align}
\mathcal{Y}_{eq}(\mathcal{G},g)&=\left\{y\in \mathbb{R}^{\mathrm d|\mathcal{E}|}: \textcolor{black}{\sum_{nj\in \mathcal E}\hat \kappa_{nj}(g_{nj}(t),y_{nj})}=0\right\},
\end{align}
\textcolor{black}{where $\hat \kappa_{ij} : \mathbb{S}^{\mathrm d-1} \times \mathbb{R}^{\mathrm d } \to \mathbb{R}^{\mathrm d}$ is an extension of the domain of $\kappa_{ij}$. The vector $y \in \mathbb{R}^{\mathrm d|\mathcal E|}$ is partitioned into $\mathrm d$-dimension blocks indexed by each edge in $\mathcal G$.}
We now describe how the sets $\mathcal X(\mathcal G,\bm{\mathrm g})$ and $\mathcal Y(\mathcal G,{g})$ are related.
\begin{Lemma}
    Assume that $g,\bm{\mathrm g} \in \mathcal{C}_{1}\bigcap \mathcal{C}_{F_B}(\mathcal{G})$.  Then  
    \begin{itemize}
        \item[i)] $\bm{\mathrm{g}} \in \mathcal{X}(\mathcal{G},\bm{\mathrm{g}})$, and
        \item[ii)] $g \in \mathcal{Y}(\mathcal{G},g)$.
    \end{itemize}
    \label{lem:permanent solution}
\end{Lemma}

\begin{proof}
    By assumption, both formations $(\mathcal G,g)$ and $(\mathcal G,\bm{\mathrm g})$ are realizable bearing formations. From the properties of the projection matrices, it follows that $P_{g_k}g_k = 0$ and $P_{\bm{\mathrm g_k}}\bm{\mathrm g_k} = 0$ for $k=1,\ldots,|\mathcal E|$, and therefore $\bm{\mathrm g} \in \mathcal X_{eq}(\mathcal G,\bm{\mathrm g}) $ and ${g} \in \mathcal Y_{eq}(\mathcal G,{g}) $. 
\end{proof}
Lemma \ref{lem:permanent solution} shows that the sets $\mathcal{X}(\mathcal{G},\bm{\mathrm{g}})$ and $\mathcal{Y}(\mathcal{G},g)$ are always non-empty.  
We would further like to understand under what conditions does the set $\mathcal{X}(\mathcal{G},\bm{\mathrm{g}})$ only contain the target formation, i.e., that there is a single equilibrium for the dynamics \eqref{sys:1DOF}.  
\begin{Lemma}\label{lem:31}
For the system \eqref{sys:1DOF}, if $\mathcal{Y}(\mathcal{G},{\bar{\mathrm g}})=\{{\bar{\mathrm g}}\}$ for \textcolor{black}{any} bearing vector ${\bar{\mathrm g}}\in \mathcal{C}_{1}\cap \mathcal{C}_{F_B}(\mathcal{G})$, then $\mathcal{X}(\mathcal{G},\bm{\mathrm{g}})=\{\bm{\mathrm{g}}\}$ for \textcolor{black}{any} target bearing vector $\bm{\mathrm{g}}\in \mathcal{C}_{1}\cap \mathcal{C}_{F_B}(\mathcal{G})$.
\end{Lemma}

\begin{proof}
The lemma is proven by contradiction. 
Assume that $\mathcal{X}(\mathcal{G},\bm{\mathrm{g})}) = \{\bm {\mathrm g} , {\bar{\mathrm g}}\}$ \textcolor{black}{(i.e., $\mathcal{X}(\mathcal{G},\bm{\mathrm{g})})$ is not a singleton)}, which means that 
\textcolor{black}{$$\sum_{nj \in \mathcal E} \kappa_{ij}(\bm{\mathrm g}_{nj},\bm{\mathrm g}_{nj}) = \sum_{ij \in \mathcal E} \kappa_{nj}({\bar{\mathrm g}}_{nj},\bm{\mathrm g}_{nj})=0.$$}

This then implies that $\bm{\mathrm g} \in \mathcal Y(\mathcal G,{\bar{\mathrm g}})$ and from Lemma \ref{lem:permanent solution} we have $\{\bm {\mathrm g}, {\bar{\mathrm g}}\} \subseteq \mathcal Y(\mathcal G,{\bar{\mathrm g}})$, leading to a contradiction. \textcolor{black}{Note that since we are considering vectors in $\mathcal{C}_{1}\cap \mathcal{C}_{F_B}(\mathcal{G})$, the maps $\tilde \kappa_{ij}$ and $\hat \kappa_{ij}$ are identical to the map $\kappa_{ij}$.}
\end{proof}

Lemmas \ref{lem:permanent solution} and \ref{lem:31} show that if $\mathcal{Y}(\mathcal{G},g)$ is a singleton, than so must be $\mathcal{X}(\mathcal{G},\bm{\mathrm{g})})$. \textcolor{black}{The first step is to} derive conditions that guarantee that this in fact happens \textcolor{black}{(i.e., the sets are singletons)}. \textcolor{black}{In this direction, we derive} an explicit characterization of the sets $\mathcal Y_{eq}(\mathcal G,g)$, $\mathcal Y_{eq}(\mathcal G,g) \cap \mathcal C_{1}$, and finally $\mathcal Y(\mathcal G,g)$.

We start by finding the elements in $\mathcal{Y}_{eq}(\mathcal{G},g)$. Define $\Tilde{P}$ as 
$$\tilde{P}=\begin{bmatrix} P_{g_{n1}} & \cdots & P_{g_{n(n-1)}} \end{bmatrix}\in \mathbb{R}^{\mathrm d \times \mathrm d (n-1)}.$$ 
Then, the equilibrium condition for the follower node $n$ in \eqref{sys:1DOF} can be expressed as
\begin{equation}
    0=-\sum_{nj \in \mathcal E} P_{g_{nj}}{\bm{\mathrm{g}}_{nj}}=-\tilde{P} {\bm{\mathrm{g}}}.
\end{equation}
It then follows that $\mathcal{Y}_{eq}(\mathcal{G},g) = \Null(\Tilde{P})$.  For formations in $\mathbb{R}^{\mathrm dn}$, it follows that $\mathrm{rk} \tilde P = \mathrm d$, and the dimension of its null-space is therefore $\mathrm d(n-2)$.  \textcolor{black}{This is due to Assumption \ref{asp:1dof_fc} which ensures that equilibrium configurations do not have any collinear bearings.}

Let $G = \mathrm{diag}\{g_{ni}\}_{i=1}^{n-1} \in \mathbb{R}^{\mathrm d(n-1) \times n-1}$ and $G^\perp = \mathrm{diag}\{g_{ni}^\perp\}_{i=1}^{n-1} \in \mathbb{R}^{\mathrm d(n-1) \times (\mathrm d-1)(n-1)}$, where $g_{ni}^\perp \in \mathbb{R}^{\mathrm d \times \mathrm d-1}$ is defined such that $\range(g_{ni}^\perp) = \Null (g_{ni}^T(t))$. It then follows that $\range(G)$ and $\range(G^\perp)$ are orthogonal subspaces, and that $\range(G)\oplus\range(G^\perp) = \mathbb{R}^{\mathrm d (n-1)}$.

The following lemma relates properties of $G$ and $G^\perp$  to the matrix  $\tilde{P}$.
\begin{Lemma}\label{lem:G_Gp}
For $G$ and $G^\perp$ defined above, the following hold:
\begin{itemize}
    \item[i)]$\Tilde{P} G=\vmathbb{0}_{\mathrm d\times n-1}$;
    \item[ii)]$\Tilde{P} G^\perp  = \begin{bmatrix}g_{n1}^\perp & \cdots & g_{n(n-1)}^\perp\end{bmatrix}\in \mathbb{R}^{\mathrm d \times (\mathrm d-1)(n-1)}$.
\end{itemize}

\end{Lemma}
\begin{proof}
The proof follows by direct construction.  For part $i$), we have 
\begin{align*}
    \tilde{P} G&= \begin{bmatrix}
        P_{g_{n1}}g_{n1} & \cdots & P_{g_{n(n-1)}}g_{n(n-1)}
    \end{bmatrix}=\vmathbb{0}_{\mathrm d\times n-1}. 
\end{align*}
Similarly, for $ii$) we have $P_{g_{ni}} \mathrm  g_{ni}^\perp=\mathrm 
 g_{ni}^\perp$ and the result follows directly.  
\end{proof}

Lemma \ref{lem:G_Gp} shows that $\range(G)\subset \Null(\Tilde{P})$. The columns of $G$ therefore can be used to  determine $n-1$ basis vectors of $\Null(\tilde{P})$, while there are $ m=\mathrm d(n-2)-n+1$ basis vectors left to be determined. These basis vectors should be orthogonal to $G$, which can be expressed by the linear combination of {{the columns of}} $G^\perp$. 

In this direction, define $N \in \mathbb{R}^{(\mathrm d-1)(n-1) \times m}$ such that $\range(N)=\Null(\tilde P G^\perp)$, 
i.e., $\Tilde{P} G^\perp N=\vmathbb{0}_{\mathrm d \times m}$, from which it follows that $\range(G^\perp N) \subset \Null(\tilde{P})$. Thus, we conclude that $$\Null(\tilde{P})=\range(G)\oplus \range(G^\perp N).$$

We are now prepared to express $\mathcal{Y}_{eq}(\mathcal{G},g)$ in terms of linear combinations of the columns of $G$ and $G^\perp N$, 
{\small\begin{equation*}
    \mathcal{Y}_{eq}(\mathcal{G},g)=\{y\in \mathbb{R}^{\mathrm d|\mathcal{E}|}:  y =G a+G^\perp N b, \, \forall \, a\in\mathbb{R}^{n-1}, \, b\in\mathbb{R}^{m} \}.
    \label{eqn:solution form_m}
\end{equation*}}
%
%
Equivalently, we can express vectors $y \in \mathcal{Y}_{eq}(\mathcal{G},g)$ in terms of its components as
%
\begin{equation}
    y_i=a_i g_{ni}+g_{ni}^\perp\left(\sum_{j=1}^{m}b_jn_{ij}\right),
    \label{eqn:solution form}
\end{equation}
where $n_{ij}\in \mathbb{R}^{\mathrm d-1}$ and
$$N = \begin{bmatrix}n_{11} & \cdots & n_{1m} \\ \vdots & \ddots & \vdots \\ n_{(n-1)1} & \cdots & n_{(n-1)m} \end{bmatrix}.$$

We now use the characterization in \eqref{eqn:solution form} to find all solutions that are also in $\mathcal C_{1}$.  In particular, we must have that for each $i$,
\begin{align*}
    \|y_i\|^2 & = \left(a_i g_{ni}+g_{ni}^\perp\left(\sum_{j=1}^{m}b_jn_{ij}\right)\right)^T\left(a_i g_{ni}+g_{ni}^\perp\left(\sum_{j=1}^{m}b_jn_{ij}\right)\right) \\
    &= a_i^2 +  \left\|\sum_{j=1}^{m}b_jn_{ij}\right\|^2 = 1.
\end{align*}
This holds since $g_{ni}^Tg_{ni}^\perp=0$ and $(g_{ni}^\perp)^Tg_{ni}^\perp = I_{\mathrm d-1}$.


%
Finally, we can consider the realizable vectors described by the form above.
\begin{Lemma}
    In the 1-to-many system, the bearing vector $\mathrm{g} \in \mathcal{Y}_{eq}(\mathcal{G},g)\cap \mathcal{C}_{1}$ is realizable if and only if $a=\vmathbb{1}_{n-1},\, b=\vmathbb{0}_{m}$. Equivalently, $\mathcal{Y}(\mathcal{G},g)=\{g\}$. 
    %
    \label{lem:1dof_realizable}
\end{Lemma}
\begin{proof}
    ($\Leftarrow$) For $a=\vmathbb{1}_{n-1},\, b=\vmathbb{0}_{m}$ we have $g=Ga+G^\perp Nb= g$. The vector $g$ corresponds to a bearing measurement, it must be realizable.

    ($\Rightarrow$) We prove this by contradiction.  Assume there exists an $a \neq \vmathbb{1}_{n-1}$ and $b \neq \vmathbb{0}_{m} $ such that the bearing vector $\bar {g} = Ga+G^\perp N b \in \mathcal C_1 \cap \mathcal C_{F_B}$. Therefore, we have that $\{\bar {g}, g\} \subseteq \mathcal Y(\mathcal G,g)$. Let $\bar{p}$ and $p$ be such that $\bar {g} = F_B(\bar{p})$ and $ {g} = F_B(p)$.
    
    Since the leader positions are fixed, we can consider the displacement between the follower agent of the two solutions described above.  Therefore, let $z=\bar p_n-p_n$, and for each $i=1,\ldots,n-1$,
    %
    %
%
\begin{equation}
    \begin{aligned}
        z_i&=\bar p_n-p_i+p_i-p_n\\
        &=\bar d_{ni} \bar {g}_{ni}-d_{ni}{g}_{ni}\\
        &=(\bar d_{ni}a_i-d_{ni})g_{ni}+g_{ni}^\perp\left(\sum_{j=1}^{m}b_jn_{ij}\bar d_{ni}\right), 
    \end{aligned}
\end{equation}
where $d_{ni}=\|p_n-p_i\|$ and $\bar d_{{n}i}=\|\bar{p}_n-p_i\|$. 
Multiplying on the left by $z_i^T$ of $g_{ni}^\perp\left(\sum_{j=1}^{m}b_jn_{ij}\right)$ gives
\begin{equation}
    \begin{aligned}
    &z_i^T g_{ni}^\perp\left(\sum_{j=1}^{m}b_jn_{ij}\right)\\
    = & \left((\bar d_{ni}a_i-d_{{n}i})g_{ni}^T+\bar d_{ni}\left(\sum_{j=1}^{m}b_jn_{ij}\right)^T(g_{ni}^\perp)^T\right)g_{ni}^\perp\left(\sum_{j=1}^{m}b_jn_{ij}\right)\\
    &= \bar d_{ni}\left(\sum_{j=1}^{m}b_jn_{ij}\right)^T\left(\sum_{j=1}^{m}b_jn_{ij}\right)
    \geq  0.
    \end{aligned}
\label{ieq:zg>=0}
\end{equation}

On the other hand, the following equation always holds, 
\begin{equation}
    \begin{aligned}
        (\vmathbb{1}_n^T \otimes I_{\mathrm d}) G^\perp Nb =\vmathbb{0}_{\mathrm d}
        \Leftrightarrow \sum_{i=1}^{n-1} \left(g_{ni}^\perp\sum_{j=1}^{m}b_jn_{ij}\right) &= \vmathbb{0}_{\mathrm d},
    \end{aligned}
    \label{eqn:sum0}
\end{equation}
since $(\vmathbb{1}_n^T \otimes I_{\mathrm d}) G^\perp=\Tilde{P}G^\perp=[g_{n1}^\perp,\cdots,g_{n(n-1)}^\perp]$. Any vector multiplied by zero is zero, which leads to

\begin{equation}
    \begin{aligned}
        z_i^T  \sum_{i=1}^{n-1} \left(g_{ni}^\perp \sum_{j=1}^{m}b_jn_{ij}\right)  = \sum_{i=1}^{n-1}\left(z_i^Tg_{ni}^\perp \sum_{j=1}^{m}b_jn_{ij}\right) &= 0.
    \end{aligned}
    \label{eqn:sum1}
\end{equation}
 From \eqref{ieq:zg>=0} and \eqref{eqn:sum1}, it can concluded that $z_i^Tg_{ni}^\perp \sum_{j=1}^{m}b_jn_{ij}\geq 0$ for all terms with $i=1,\cdots,n-1$ and their sum equals to zero. Thus, every single term should be exactly zero, which means $(\bar d_{ni}a_i-d_{{n}i})g_{ni}+g_{ni}^\perp\left(\sum_{j=1}^{m}b_jn_{ij}\bar d_{ni}\right)=0$. This equation then directly leads to $a_i=1,\forall i=1,\ldots,n-1$ and $b_j=0,\forall j=1,\ldots,m$. 
\end{proof}
Lemma \ref{lem:1dof_realizable} shows that $\mathcal{Y}(\mathcal{G},g)=\{g\}$. It then follows from Lemma $\ref{lem:31}$ that $\mathcal{X}(\mathcal{G},\bm{\mathrm{g}})=\{\bm{\mathrm{g}}\}$. 
We now show that the equilibrium position of the follower agent, $\mathrm p_n$ can be uniquely determined from the initial conditions (target positions) of the leaders and the bearing measurements.
\begin{proposition}
    Let Assumption \ref{asp:1dof_fc} hold for the target bearing formation $(\mathcal G,\bm{\mathrm g})$ and assume that $p_i(0)=\mathrm p_i$ for a configuration $\mathrm p \in F_B^{-1}(\bm{\mathrm g})$. Then 
    \begin{align}\label{1tomany_eq}
    \mathrm{p}_n=\left(\sum_{nj\in\mathcal{E}} P_{\bm{\mathrm g}_{nj}}\right)^{-1}\left(\sum_{nj\in\mathcal{E}}  P_{\bm{\mathrm g}_{nj}}\mathrm p_j\right)
    \end{align}
    is an equilibrium of \eqref{sys:1DOF}.
    \label{lem:tp_1dof}
\end{proposition}
Proposition \ref{lem:tp_1dof} translates the equilibrium from conditions on the bearing measurements to the position of agent $n$. Before proving it, we present a useful lemma related to the properties of the projection matrix.
\begin{Lemma}\label{lem.proj}
Let $x,y \in \mathbb{R}^{\mathrm d}$ be two non-parallel vectors.  Then $P_x+P_y$ is invertible.
\end{Lemma}
\begin{proof}
    From the definition of the projection matrix \eqref{projmat}, it follows that $\Null(P_x)=\Span\{x\}$ and $ \Null(P_y)=\Span\{y\}$. Since $x$ is not parallel with $y$, the subspace 
    $$\Null(P_x)\cap \Null(P_y)=\Span\{x\}\cap \Span\{y\}=\{0\}.$$ 
    In addition, the projection matrix is  positive semi-definite. {The kernel  of two positive semi-definite matrices is the intersection of the kernel space of these two matrices} (i.e., $\Null(P_x+P_y)=\Null(P_x)\bigcap \Null(P_y)$). Thus, the space $\Null(P_x+P_y)=\{0\}$, implying that $P_x+P_y$ is invertible.
\end{proof}

\begin{proof}[Proof of Proposition \ref{lem:tp_1dof}]
For the configuration $\mathrm p$, if $F_B(\mathrm p)=\bm{\mathrm g}$, it always holds that
        
    $$ P_{\bm{\mathrm g_{nj}}} (\mathrm p_n - \mathrm p_j)=0,\forall nj\in\mathcal{E}.$$
Furthermore, the sum
    $$\sum_{nj\in \mathcal E}P_{\bm{\mathrm g_{nj}}}(\mathrm p_n - \mathrm p_j)=0 $$
    must also hold.
    Rearranging terms leads to
    $$\left(\sum_{nj\in\mathcal{E}}P_{\bm{\mathrm g_{nj}}}\right)\mathrm p_n = \sum_{nj\in\mathcal{E}} P_{\bm{\mathrm g_{nj}}}\mathrm p_j.$$    
    Assumption \ref{asp:1dof_fc} requires $\bm{\mathrm{g}}_{ni}\neq \bm{\mathrm{g}}_{nj}\ \forall   i,j=1,\ldots,n-1$, and it follows from Lemma \ref{lem.proj} that $\left(\sum_{nj\in\mathcal{E}} P_{\bm{\mathrm g_{nj}}}\right)$ is invertible. As a result, the target position $\mathrm{p}_n$ in \eqref{1tomany_eq} holds.
\end{proof}

The last step is to determine the stability of equilibrium. 

\begin{theorem}\label{thm:1-to-many}
Let Assumption \ref{asp:1dof_fc} hold for some target bearing configuration $(\mathcal G,\bm{\mathrm g})$ and assume that $p_i(0)=\mathrm p_i$ for a $\mathrm p \in F_B^{-1}(\bm{\mathrm g})$. Then the point $\mathrm p_n$ defined in \eqref{1tomany_eq} is a globally exponentially stable equilibrium for the dynamics in \eqref{sys:1DOF}.
\end{theorem}
\begingroup
\color{black}
\begin{proof}
Consider the Lyapunov function $V(p_n(t))=\frac{1}{2}\|p_n(t)-\mathrm p_n\|^2$. Then
 \begin{align}\label{eqn:vdot_pt1}
    \dot V(p_n(t))&=(p_n(t)-\mathrm p_n)^T \dot p_n(t) \nonumber\\
   &=-\sum_{nj\in\mathcal{E}} (p_n(t)-\mathrm p_n)^T P_{g_{nj}(t)}\mathbf{g}_{nj}\nonumber\\
   &=-\sum_{nj\in\mathcal{E}} (p_n(t)-\mathrm p_j+\mathrm p_j-\mathrm p_n)^T P_{g_{nj}(t)}\mathbf{g}_{nj}\nonumber\\
   &=-\sum_{nj\in\mathcal{E}} \mathbf{d}_{nj} \mathbf{g}_{nj}^T P_{g_{nj}(t)}\mathbf{g}_{nj}  \leq 0.
   \end{align}
    In above, $\mathbf{d}_{nj} = \|\mathrm p_j-\mathrm p_n\| > 0$ (by Assumption \ref{asp:1dof_fc}), and we recall that $P_{g_{nj}(t)}(p_n(t)-\mathrm p_j) = 0$. Therefore, $\|p_n(t) - \mathrm p_n\|$ is bounded and non-increasing along trajectories. Since the leader positions are also fixed, it follows that $\|p_n(t)-\mathrm p_j\|$ is also bounded along trajectories for each $j=1,\ldots,n-1$. 
    
    We now continue the derivation to show exponential stability.  First, recall that for $x,y \in \mathbb{S}^{\mathrm d-1}$, $x^TP_yx = y^TP_xy$.  We use this below continuing from \eqref{eqn:vdot_pt1}.
   \begin{align}\label{eqn:vdot_pt2}
   \dot V(p_n(t))&= -\sum_{nj\in\mathcal{E}} \mathbf{d}_{nj} \mathbf{g}_{nj}^T P_{g_{nj}(t)}\mathbf{g}_{nj}\nonumber \\
   &\hspace{-1cm}=-\sum_{nj\in\mathcal{E}} \mathbf{d}_{nj} {g}_{nj}^T(t) P_{\mathbf{g}_{nj}}{g}_{nj}(t)\nonumber\\
   &\hspace{-1cm}=-\sum_{nj\in\mathcal{E}} \frac{\mathbf{d}_{nj}}{d_{nj}^2(t)} (\mathrm p_j-p_n(t))^T P_{\mathbf{g}_{nj}}(\mathrm p_j-p_n(t)).
   \end{align}
   Now note that $(\mathrm p_j-p_n(t))=(\mathrm p_j-\mathrm p_n+\mathrm p_n-p_n(t))$ and $P_{\mathbf{g}_{nj}}(\mathrm p_j-\mathrm p_n) = 0$.
From earlier we have that both $\mathbf{d}_{nj}$ and $d_{nj}(t)$ are bounded, and therefore 
$$\frac{\mathbf{d}_{nj}}{d_{nj}^2(t)}\geq \min_j \left(\frac{\mathbf{d}_{nj}}{\sup_td_{nj}^2(t)} \right)= \gamma, \, j=1,\ldots,n-1.$$

Continuing from \eqref{eqn:vdot_pt2} and using the above observations, we have
\begin{align}\label{eq:vdot_pt3}
\dot V(p_n(t))&= -\sum_{nj\in\mathcal{E}} \frac{\mathbf{d}_{nj}}{d_{nj}^2(t)} (\mathrm p_j-p_n(t))^T P_{\mathbf{g}_{nj}}(\mathrm p_j-p_n(t))\nonumber \\ 
&\hspace{-1cm}\leq - (p_n(t)-\mathrm p_n)^T \left(\gamma\sum_{nj\in\mathcal{E}} P_{\mathbf{g}_{nj}}\right)(p_n(t)-\mathrm p_n)
\end{align}
Let $M = \gamma \sum_{nj\in\mathcal{E}}P_{\mathbf{g}_{nj}}$.  By Assumption \ref{asp:1dof_fc}, the target bearings $\mathbf{g}_{ni}$ and $\mathbf{g}_{nj}$ are not parallel for any $i,j\in \{1,\ldots,n-1\}$, and therefore $M \succ 0$ (see Lemma \ref{lem.proj}). We finally conclude that $$\dot V(p_n(t)) \leq -  \lambda_{\min}(M) V(p_n(t)),$$ 
where $\lambda_{\min}(M)$ denotes the smallest eigenvalue of $M$. Thus, $\mathrm p_n$ is a globally exponentially stable equilibrium.
\end{proof}
\endgroup

For the 1-to-many system with at least 2 leaders, the follower $p_n(t)$  converges to the target position specified in \eqref{1tomany_eq} exponentially fast. \textcolor{black}{The following corollary shows that the convergence rate strictly increases as the number of leaders increases.
\begin{cor}\label{cor_incleader}
    Assume the same conditions as in Theorem \ref{thm:1-to-many}.  The rate of convergence of \eqref{sys:1DOF} strictly increases with the number of leaders.
\end{cor}
\begin{proof}
    Let $M= \gamma
    \sum_{nj\in\mathcal{E}}P_{\mathbf{g}_{nj}}$ (as in the proof of Theorem \ref{thm:1-to-many}) and assume a new leader is added, say with index $\ell$.  Then $M$ is updated to $M'(t) = M + \frac{\mathbf{d}_{n\ell}}{d_{n\ell}^2(t)}P_{\mathbf{g}_{n\ell}}$. The matrix \(  \frac{\mathbf{d}_{n\ell}}{d_{n\ell}^2(t)}P_{\mathbf{g}_{n\ell}} \) is symmetric and positive semi-definite, and therefore $\min {\lambda}(M'(t)) \geq \min{\lambda}(M)$. Let \( v \) be a unit eigenvector associated with the smallest eigenvalue \( \min{\lambda}(M) \). Then
$$v^T M'(t) v = v^T M v + v^T \left( \frac{\mathbf{d}_{n\ell}}{d_{n\ell}^2(t)}P_{\mathbf{g}_{nj}} \right) v.$$
Since $v$  is not parallel to $\mathbf{g}_{n\ell}$ by Assumption~\ref{asp:1dof_fc}, we have
$$v^T\left( \frac{\mathbf{d}_{n\ell}}{d_{n\ell}^2(t)} P_{\mathbf{g}_{n\ell}}\right) v =  \frac{\mathbf{d}_{n\ell}}{d_{n\ell}^2(t)}\left(\|v\|^2 - (v^T \mathbf{g}_{n\ell})^2\right) > 0,$$
implying $v^T M'(t) v > v^T M v = \min{\lambda}(M)$.  It follows that 
$\min{\lambda}(M'(t)) > \min{\lambda}(M)$ which holds for all $t$.
\end{proof}
}

We conclude this section by examining a special configuration of the 1-to-many system that admits only an unstable equilibrium.  We consider initial conditions for the leaders that satisfy $p_i(0) = \bar{\mathrm p}_i, \,i=1,\ldots, n-1$ for a $\bar{\mathrm p} \in F_B^{-1}(-\bm{\mathrm g})$ \textcolor{black}{(i.e., the leaders are not in a configuration corresponding to the desired bearings $\bm{\mathrm g}$)}. In this direction, we first introduce the notion of symmetric configurations. 
%


\begin{definition}\label{def.symmetric}
    Two configurations $p=[p_1^T\  \cdots\  p_n^T]^T$ and $q=[q_1^T\  \cdots\  q_n^T]^T$ are \emph{ symmetric with respect to $c$} if
    $$p_i-c=c-q_i,\; i=1,\ldots,n-1. $$
\end{definition}

\begin{proposition}\label{lem:unstable_1DOF_IC}
    For the target formation $(\mathcal{G},\bm {\mathrm{g}})$, consider two configurations $\mathrm p,\  \bar {\mathrm p} \in \mathbb{R}^{|\mathcal{V}|\mathrm {d}}$. If $\mathrm p$ and $\bar {\mathrm p}$ are symmetric with respect to some point $c \in \mathbb{R}^{\mathrm d}$, and $\mathrm p \in F_B^{-1}(\bm{\mathrm{g}})$, then $\bar{\mathrm p} \in F_B^{-1}(-\bm{\mathrm{g}})$.
\end{proposition}

\begin{proof}
From Definition \ref{def.symmetric}, it follows that
$$\mathrm p_i - c = c- \bar{\mathrm p}_i,$$
or equivalently that $\bar{\mathrm p}_i = 2c-\mathrm p_i$.  Since $\mathrm p \in F_B^{-1}(\bm{\mathrm{g}})$, we also have that
\begin{equation*}
    \frac{\mathrm p_i-\mathrm p_n}{\|\mathrm p_i-\mathrm p_n\|}=\bm{\mathrm g}_{ni}, \, i=1,\ldots,n-1.
\end{equation*}
Consider now the point $\bar {\mathrm p}_n$.  Since $\mathrm p$ and $\bar {\mathrm p}$ are symmetric with respect to $c$, it follows that $\bar{\mathrm p}_n = 2c-\mathrm p_n$.  We then have
\begin{align}\label{eqn:sym_IC}
    \frac{\bar{\mathrm p}_i- \bar{\mathrm p}_n}{\|\bar{\mathrm p}_i- \bar{\mathrm p}_n\|} &=\frac{2c-\mathrm p_i - (2c-\mathrm p_n)}{\|2c-\mathrm p_i - (2c-\mathrm p_n) \|}\nonumber\\
    &= -\frac{\mathrm p_i - \mathrm p_n}{\|\mathrm p_i - \mathrm p_n \|} = - \bm{\mathrm{g}_{ni}}.
\end{align}
    
%
Since equation \eqref{eqn:sym_IC} holds for every leader $i=1,\ldots, n-1$, it follows that $\bar{\mathrm p}\in F_B^{-1}(-\bm{\mathrm g})$
\end{proof}

\color{black}

\begin{Lemma}\label{lem:unstable_1dof}
    Let Assumption \ref{asp:1dof_fc} hold for the target bearing formation $(\mathcal G,\bm{\mathrm g})$ and assume that $p_i(0)=\bar{\mathrm p}_i$ for a configuration $\bar{\mathrm p} \in F_B^{-1}(-\bm{\mathrm g})$. Then 
    \begin{align*}
    \mathrm{p}_n=\left(\sum_{nj\in\mathcal{E}} P_{\bm{\mathrm g}_{nj}}\right)^{-1}\left(\sum_{nj\in\mathcal{E}}  P_{\bm{\mathrm g}_{nj}}\bar{\mathrm p}_j\right)
    \end{align*}
    is the only equilibrium of \eqref{sys:1DOF} and it is unstable.  
\end{Lemma}

\begin{proof}
    First, recall that the control \textcolor{black}{terms $\kappa_{nj}(g_{nj},\bm{\mathrm g}_{nj})$ are} linear in the target bearing $\bm{\mathrm g}$.  Therefore, it follows that \textcolor{black}{$\kappa_{nj}(g_{nj},\bm{\mathrm{g}}_{nj})=-\kappa_{nj}(g_{nj},-\bm{\mathrm{g}}_{nj})$}, and the BOFC for the bearing formation $(\mathcal G, \bm{\mathrm{g}})$ has the same equilibrium point as the bearing formation $(\mathcal G, -\bm{\mathrm{g}})$. Furthermore, the dynamics of the follower satisfies
    \textcolor{black}{\begin{align*}\dot p_n(t) &=\sum_{nj\in \mathcal E}\kappa_{nj}(g_{nj},\bm{\mathrm{g}}_{nj}) \\&=\sum_{nj\in \mathcal E} -\kappa_{nj}(g_{nj},-\bm{\mathrm{g}}_{nj}) = \sum_{nj\in \mathcal E}P_{\mathrm g_{nj}(t)}\bm{\mathrm g}_{nj}.\end{align*}}
    The equilibrium condition can now be verified using the same arguments as in  Proposition \ref{lem:tp_1dof}.  Using the same Lyapunov function construction as in Theorem \ref{thm:1-to-many} it is straightforward to verify that $\dot V(p_n(t))> 0$. Applying Chetaev instability theorem we can conclude this equilbirum is unstable.

\end{proof}

A key point of this result shows that symmetric configurations for the leader initial conditions correspond to either  stable or unstable trajectories. This result may not be suprising as we are choosing initial conditions for the leaders that do not correspond to desired bearing measurements.  Nevertheless, this characterization will be important to establish equilibria and stability conditions for more general LFF graphs in the sequel.

\subsection{BOFC with Ordered LFF Graphs}
\label{c4}

In the previous section, we outlined an analysis approach for the one-to-many BOFC.  This turns out to be a central idea when trying to generalize the LFF structures.  In particular, we aim next to relax the assumption in Definition \ref{def:HCLFF} that each follower agent must have exactly 2 outgoing edges. We call such graphs \emph{ordered LFF graphs}. An example of an ordered LFF graph is showed in Figure \ref{fig.odLFF}, with the leader and first-follower as nodes $1,2$, respectively. Note that all edges are forward edges but it is possible for some nodes to have more than two outgoing edges, \textcolor{black}{which is the main difference from the LFF graph.}





\begin{definition}
    A directed graph  is an \emph{ordered LFF graph} (OLFF) if
    \begin{itemize}
        \item[$i$)] there is a vertex with no outgoing edges, denoted as the \emph{leader}, assigned the label $v_1$; 
    \item[$ii$)] there is a vertex with only one outgoing edge pointing to the leader, denoted as the \emph{first follower} assigned the label $v_2$;
        \item[$iii$)] every vertex $v_i$ other than the leader and the first follower has at least two outgoing edges, pointing to vertices with indices smaller that $i$.
    \end{itemize}
\label{prop:1}
\end{definition}

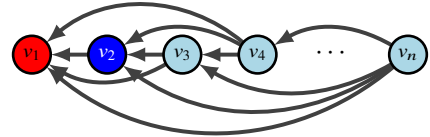
\begin{figure}[h]
        \centering
        \begin{tikzpicture}
        \Vertex[size=.5,label=$v_1$,color=red,x=0,y=0]{v1}
        \Vertex[size=.5,label=$\color{white}{v_2}$,color=blue,x=1,y=0]{v2}
        \Vertex[size=.5,label=$v_3$,x=2,y=0]{v3}
        \Vertex[size=.5,label=$v_4$,x=3,y=0]{v4}
        \Vertex[size=.5,label=$v_{n}$,x=5,y=0]{vi}
        \node at (4,0) {$\cdots$};
        \Edge[Direct](v2)(v1)
        \Edge[Direct,bend=30](v3)(v1)
        \Edge[Direct](v3)(v2)
        \Edge[Direct,bend=-40](v4)(v1)
        \Edge[Direct,bend=-30](v4)(v2)
        \Edge[Direct](v4)(v3)
        \Edge[Direct,bend=40](vi)(v1)
        \Edge[Direct,bend=35](vi)(v2)
        \Edge[Direct,bend=30](vi)(v3)
        \Edge[Direct,bend=-30](vi)(v4)
        \end{tikzpicture}
        \caption{An example of an ordered LFF graph.}\label{fig.odLFF}
    \end{figure}

\textcolor{black}{Returning to the example in Figure \ref{fig:ex1_tf_2}, we see that this graph is OLFF. }When considering the BOFC for directed graphs \eqref{eqn:law_di} over ordered LFF graphs, we reveal that it is still possible to express the dynamics in a cascade form.  Indeed, for OLFF graphs, the dynamics for each agent can be expressed as
\begin{align}\label{sys:p1}
   \begin{cases}
       \dot p_1(t) &= 0 \\
\dot p_2(t) &= \textcolor{black}{-P_{g_{21}(t)}\bm{\mathrm{g}_{21}}} \\
\dot p_i(t) &= -\sum_{ij\in \mathcal E}\kappa_{ij}(g_{ij}(t),\bm{\mathrm g}_{ij}), \, i=3,\ldots,n
\end{cases},
\end{align}
where we have that agent $i$ only requires information from agents with indices $j < i$.

Similar to Assumption \ref{asp:1dof_fc}, an assumption on the target bearing formation is required for  Problem (\ref{prob.formationcontrol}).

\begin{asp}
    The target bearing formation $(\mathcal{G},\bm{\mathrm{g}})$ is realizable. Furthermore, there exists a configuration $\mathrm p \in F_B^{-1}(\bm{\mathrm g})$ such that $\mathrm p_i, \mathrm p_j$ and $\mathrm p_k$ are not collinear for any $v_j,v_k \in \mathcal{N}_i$.
    \label{asp:bfc_fc}
\end{asp}

\textcolor{black}{Once a target bearing formation is specified, the absolute position and scale of the formation remain to be fixed. In the ordered LFF graph structure, the absolute position of the leader $p_1(0)$ anchors the formation’s translational degree of freedom, while the initial distance $d_{21}(0)$ between the leader and the first follower fixes the formation’s scale. Together, these anchor conditions determine a unique realization of the target formation in $\mathbb{R}^{\mathrm dn}$. This relationship is formalized in the following lemma.}

\begin{Lemma}\label{lem:uni_tc}
    (Uniqueness of target configuration). 
    \textcolor{black}{Given a target bearing formation $(\mathcal{G},\bm{\mathrm g})$ with an ordered LFF graph $\mathcal G$ satisfying Assumption 2, and given the initial position of the leader agent $p_1(0)$ and the initial distance $d_{21}(0)$ between the leader and the first follower, the target configuration $\mathrm p \in \mathbb{R}^{dn}$ consistent with $(\mathcal{G},\bm{\mathrm g})$ is uniquely determined.}
    More specifically, $\mathrm p_i$ is calculated iteratively by:
    \begin{align*}
        \mathrm p_1&= p_1(0)\\
        \mathrm p_2&= p_1(0) - d_{21}(0) \bm{\mathrm g}_{21}\\
        \mathrm p_i&= \left(\sum_{ij\in\mathcal{E}} P_{\bm{\mathrm g}_{ij}}\right)^{-1}\left(\sum_{ij\in\mathcal{E}}  P_{\bm{\mathrm g}_{ij}}\mathrm p_j\right),\, i=3,\ldots,n.
    \end{align*}
\end{Lemma}

The proof of the lemma is straight forward, which can be derived from \cite[Lemma 1]{trinh2018bearing} and {P}roposition \ref{lem:tp_1dof}.


\begin{theorem}\label{the:pp1}
Let Assumption \ref{asp:bfc_fc} hold for some target bearing configuration $(\mathcal G,\bm{\mathrm g})$. If $\mathcal G$ is an ordered LFF graph, then the bearing-only formation control \eqref{eqn:law_di} satisfies $$\lim_{t\to \infty} p(t)=\mathrm p \in F_{B}^{-1}(\bm{\mathrm g})$$ 
\textcolor{black}{for almost all initial conditions, and the convergence rate is exponential.}
\end{theorem}

\begin{proof}
Given the cascade structure shown in \eqref{sys:p1}, we can analyze the equilibria and stability of each agent successively.  Denote by $\bf{\mathrm p}$ the equilibria of \eqref{sys:p1}. Following \cite{trinh2018bearing}, it is straightforward to verify that ${\mathrm p}_1=p_1(0)$ and that there are two possible equilibrium configurations for the first follower, which are denoted as ${{\mathrm p}_{2a}}={{\mathrm p}}_1 - d_{21}(0){\bm{\mathrm g}}_{21}$ and ${{\mathrm p}_{2b}}={{\mathrm p}}_1 + d_{21}(0){\bm{\mathrm g}}_{21}$. \textcolor{black}{From Lemma \ref{lem:uni_tc}, it follows that $\mathrm p_{2a}$ is exactly the position of the unique target configuration $\mathrm p$. In addition, the equilibria $\mathrm p_{2b}$ is symmetric to $ \mathrm p_{2a}$ with respect to the point $\mathrm p_1$.}

The dynamics of agent 3 must follow
$$\dot p_3=u_3(p_1,p_2,p_3)=\sum_{j=1}^2\kappa_{3j}(g_{3j},\mathbf g_{3j}).$$
 We may consider the dynamics of agent 3 as a 1-to-many system where the leader agents have positions $(\mathrm p_1,\mathrm p_{2a})$ or $(\mathrm p_1,\mathrm p_{2b})$. We note that $(\mathrm p_1,\mathrm p_{2a})$ satisfies Assumption \ref{asp:bfc_fc}, while $(\mathrm p_1,\mathrm p_{2b})$ does not since it symmetric with respect to $\mathrm p_1$. 
Using Proposition \ref{lem:tp_1dof} and Lemma \ref{lem:unstable_1dof}, we conclude that 
$$\mathrm p_{3} = (P_{\bm{\mathrm{g}}_{31}}+P_{\bm{\mathrm{g}}_{32}})^{-1}(P_{\bm{\mathrm{g}}_{31}}{\mathrm p}_1+P_{\bm{\mathrm{g}}_{32}}{\mathrm p}_{2a})$$
corresponds to the stable equilibrium position of agent 3. Lemma \ref{lem:unstable_1dof} is used to show that the equilibrium point $(P_{\bm{\mathrm{g}}_{31}}+P_{\bm{\mathrm{g}}_{32}})^{-1}(P_{\bm{\mathrm{g}}_{31}}{\mathrm p}_1+P_{\bm{\mathrm{g}}_{32}}{\mathrm p}_{2b})$ is unstable.

For the other agents, all of them have at least two forward edges and the analysis is similar to that for agent 3.  Thus, we proceed by induction.
Consider an agent  $i \in \{3,\ldots,n\}$. Its equilibrium analysis can be separated into studying the systems
\begin{align*}
     \dot p_{ia} = u_i(p_1={\mathrm p}_1,p_2={\mathrm p}_{2a},\cdots,p_{i-1}=\mathrm p_{(i-1)a},p_{ia})\\
   \dot p_{ib} = u_i(p_1={\mathrm p}_1,p_2={\mathrm p}_{2b},\cdots,p_{i-1}=\mathrm p_{(i-1)b},p_{ib})
\end{align*}

The system ($i$a) is the 1-to-many system with leaders' initial condition satisfying Assumption \ref{asp:bfc_fc}.
The equilibria is therefore 
$${\mathrm p}_{ia}=\left(\sum_{j\in\mathcal{N}_i} P_{\bm{\mathrm g}_{ij}}\right)^{-1}\left(\sum_{j\in\mathcal{N}_i} P_{\bm{\mathrm g}_{ij}}\mathrm p_{ja}\right).$$

The system ($i$b) is the 1-to-many system with leaders' initial conditions that are symmetric with respect to $\mathrm p_1$, and therefore Assumption \ref{asp:bfc_fc} does not hold. 
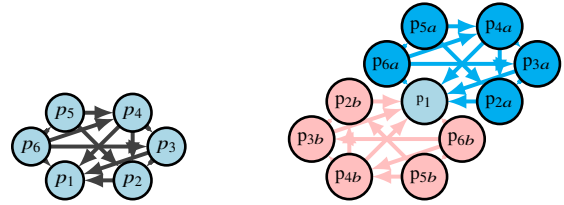
\begin{figure}[h]
    \centering
    \begin{subfigure}[t]{.3\linewidth}
    \centering
    \begin{tikzpicture}[scale=0.45]
    \Vertex[size=.5,label=$p_1$,x=0,y=0]{v1}
    \Vertex[size=.5,label=$p_2$,x=2,y=0]{v2}
    \Vertex[size=.5,label=$p_3$,x=3,y=1]{v3}
    \Vertex[size=.5,label=$p_4$,x=2,y=2]{v4}
    \Vertex[size=.5,label=$p_5$,x=0,y=2]{v5}
    \Vertex[size=.5,label=$p_6$,x=-1,y=1]{v6}
    \Edge[Direct](v2)(v1)
    \Edge[Direct](v3)(v1)
    \Edge[Direct](v3)(v2)
    \Edge[Direct](v4)(v1)
    \Edge[Direct](v4)(v2)
    \Edge[Direct](v4)(v3)
    \Edge[Direct](v5)(v2)
    \Edge[Direct](v5)(v4)
    \Edge[Direct](v6)(v1)
    \Edge[Direct](v6)(v3)
    \Edge[Direct](v6)(v4)
    \Edge[Direct](v6)(v5)
    \end{tikzpicture}
    \caption{The target formation.}
    \label{fig:tg_eq_type}
    \end{subfigure}
    \begin{subfigure}[t]{.65\linewidth}
        \centering
    \begin{tikzpicture}[scale=1]
    \Vertex[size=.6,label=\tiny{${\mathrm p}_1$},x=0,y=0]{v1}
    \Vertex[size=.6,label=${\mathrm p}_{2a}$,color=cyan,x=1,y=0]{v2a}
    \Vertex[size=.6,label=${\mathrm p}_{3a}$,color=cyan,x=1.5,y=0.5]{v3a}
    \Vertex[size=.6,label=$\mathrm p_{4a}$,color=cyan,x=1,y=1]{v4a}
    \Vertex[size=.6,label=$\mathrm p_{5a}$,color=cyan,x=0,y=1]{v5a}
    \Vertex[size=.6,label=$\mathrm p_{6a}$,color=cyan,x=-0.5,y=0.5]{v6a}
    \Vertex[size=.6,label=${\mathrm p}_{2b}$,color=pink,x=-1,y=0]{v2b}
    \Vertex[size=.6,label=${\mathrm p}_{3b}$,color=pink,x=-1.5,y=-0.5]{v3b}
    \Vertex[size=.6,label=$\mathrm p_{4b}$,color=pink,x=-1,y=-1]{v4b}
    \Vertex[size=.6,label=$\mathrm p_{5b}$,color=pink,x=0,y=-1]{v5b}
    \Vertex[size=.6,label=$\mathrm p_{6b}$,color=pink,x=0.5,y=-0.5]{v6b}
    \Edge[Direct,color=cyan](v2a)(v1)
    \Edge[Direct,color=cyan](v3a)(v1)
    \Edge[Direct,color=cyan](v3a)(v2a)
    \Edge[Direct,color=cyan](v4a)(v1)
    \Edge[Direct,color=cyan](v4a)(v2a)
    \Edge[Direct,color=cyan](v4a)(v3a)
    \Edge[Direct,color=cyan](v5a)(v2a)
    \Edge[Direct,color=cyan](v5a)(v4a)
    \Edge[Direct,color=cyan](v6a)(v1)
    \Edge[Direct,color=cyan](v6a)(v3a)
    \Edge[Direct,color=cyan](v6a)(v4a)
    \Edge[Direct,color=cyan](v6a)(v5a)
    \Edge[Direct,color=pink](v2b)(v1)
    \Edge[Direct,color=pink](v3b)(v1)
    \Edge[Direct,color=pink](v3b)(v2b)
    \Edge[Direct,color=pink](v4b)(v1)
    \Edge[Direct,color=pink](v4b)(v2b)
    \Edge[Direct,color=pink](v4b)(v3b)
    \Edge[Direct,color=pink](v5b)(v2b)
    \Edge[Direct,color=pink](v5b)(v4b)
    \Edge[Direct,color=pink](v6b)(v1)
    \Edge[Direct,color=pink](v6b)(v3b)
    \Edge[Direct,color=pink](v6b)(v4b)
    \Edge[Direct,color=pink](v6b)(v5b)
    \end{tikzpicture}
    \caption{The equilibrium configurations $\mathrm p_a$ and $\mathrm p_b$ determined from the target formation \ref{fig:tg_eq_type}.}
    \end{subfigure}
    \caption{Examples of two symmetric configurations with respect to the point $\mathrm p_1$.}
    \label{fig:eq_type}
\end{figure}
As illustrated in Figure \ref{fig:eq_type}, the configuration $\mathrm p_a$ satisfies all the target bearing constraints $\bm{\mathrm{g}}$. The configuration $\mathrm p_b$ satisfies the bearing constraint $-\bm{\mathrm{g}}$ is symmetric to configuration $\mathrm p_a$ with respect to $\mathrm p_1$.

The last step is to prove the stability of the equilibrium. Recall the bearing-only formation control system with the sensing graph described by ordered LFF graph stated in equation \eqref{sys:p1} is in the form of a cascade system. 
Firstly, for the subsystem
\begin{align*}
    \dot p_1=0,
\end{align*}
the equilibrium ${\mathrm p}_1=p_1(0)$ is stable. 

For the subsystem $\dot p_2 = u_2(p_1={\mathrm p}_1,p_2)$, it has been showed that ${\mathrm p}_{2a}$ is an almost GAS equilibria (\cite{trinh2018bearing}). With the stability theorem of cascade systems, we conclude that for the subsystem
\begin{align*}
    \begin{bmatrix}
        \dot p_1\\ \dot p_2
    \end{bmatrix}=
    \begin{bmatrix}
            u_1(p_1) \\ u_2 (p_1,p_2)
        \end{bmatrix},
\end{align*}
the equilibrium $({\mathrm p}_1,{\mathrm p}_{2a})$ is almost GAS.

Next, we find that ${\mathrm p}_{3a}$ is the GAS equilibrium for the 1-to-many system $\dot p_3=u_3 (p_1={\mathrm p}_1,p_2={\mathrm p}_{2a},p_3)$. 
%

The procedure can be repeated for the remaining subsystems $\dot p_i=u_i(p_1,\ldots,p_i)$ for $i\in\{4,\ldots,n\}$, concluding that almost global asymptotic stability for the BOFC system \eqref{sys:p1} with ordered LFF sensing graph. 

Finally, as a consequence of Theorem \ref{thm:1-to-many}, the convergence rate of each subsystem is exponential, with the Lyapunov exponent dependent on matrix $M(t)$ associated with that subsystem (see the proof of Theorem \ref{thm:1-to-many}).  Let $\lambda(M_i)$ denote this exponent.  Then the BOFC system \eqref{sys:p1} with ordered LFF sensing graph also converges exponentially fast with convergence rate $\min_{i} \lambda(M_i)$.

\end{proof}

An immediate corollary of above is that the convergence rate increases with the number of forward edges.
\begin{cor}
    Consider two formations $(\mathcal G_1,\bm{\mathrm g}_1)$ and $(\mathcal G_2,\bm{\mathrm g}_2)$ where $\mathcal G_1 \subset \mathcal G_2$ and $\mathcal G_1,\mathcal G_2$ are both ordered LFF graphs.  Then the convergence rate of the BOFC system \eqref{sys:p1} for the target formation $(\mathcal G_2,\bm{\mathrm g}_2)$ is faster than  that of the target formation $(\mathcal G_1,\bm{\mathrm g}_1)$.
\end{cor}
\begin{proof}
\textcolor{black}{The proof follows the same logic Corollary \ref{cor_incleader}.}
\end{proof}

\textcolor{black}{As a final remark, we note that the presence of a stationary leader agent in both the LFF and ordered LFF graph structures effectively fixes the translational degree of freedom in the formation. Consequently, convergence of the bearings in these frameworks also guarantees convergence of the positions $p(t)$ to a unique equilibrium configuration, determined by the initial positions of the leader and first follower. This contrasts with more general bearing-only control scenarios, where translational or scaling ambiguities may persist.}

\section{Simulation Results}\label{sec.simulation}

In this section, we present some numerical simulations to illustrate the main results of this work, in addition to an example highlighting an open problem.

\subsection{One-to-Many BOFC Example}

\begin{figure}[b]
    \centering
    \begin{subfigure}[b]{0.48\linewidth}
        \centering
        \includegraphics[width=\textwidth]{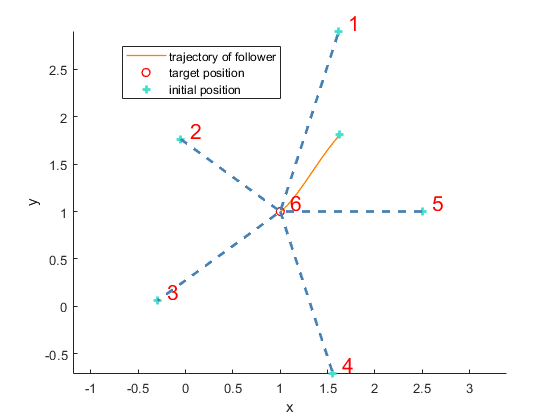}
        \caption{System trajectories.}\label{fig:sim_1m_5}
    \end{subfigure}
    \begin{subfigure}[b]{0.48\linewidth}
        \centering
        \includegraphics[width=\textwidth]{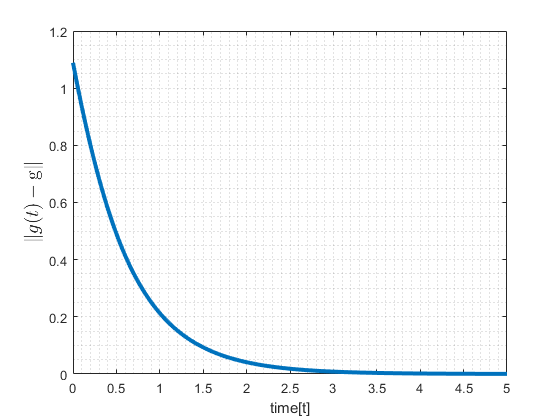}
        \caption{Bearing error along the system trajectories.}\label{fig:sim_1m_5_err}
    \end{subfigure}
    \caption{A 1-to-many system with 5 leaders.}
    \label{fig:1DOF_sim}
\end{figure}

Consider $6$ agents in the one-to-many system configuration. In this example, the desired bearings are $\bm{\mathrm{g}}_{61}=[0.309,0.951]^T$, $\bm{\mathrm{g}}_{62}=[-0.809,0.588]^T$, $\bm{\mathrm{g}}_{63}=[-0.809,-0.588]^T$, $\bm{\mathrm{g}}_{64}=[0.309,-0.951]^T$, and $\bm{\mathrm{g}}_{65}=[1,0]^T$. The leaders are placed at $\mathrm p_1=[1.618,2.902]^T$, $\mathrm p_2=[-0.051,1.764]^T$, $\mathrm p_3=[-0.294,0.060]^T$, $\mathrm p_4=[1.556,-0.712]^T$, $\mathrm p_5=[2.500,0]^T$, which can be verified to satisfy Assumption \ref{asp:1dof_fc}.
The initial position for the follower agent is chosen randomly as $p_6(0)=[1.6254,1.8106]^T$. 

According to Proposition \ref{lem:tp_1dof}, equilibrium position for the follower agent can be calculated as 
$$\mathrm p_6=\left(\sum_{j=1}^5 P_{\bm{\mathrm g}_{6j}}\right)^{-1}\left(\sum_{j=1}^5  P_{\bm{\mathrm g}_{6j}}\mathrm p_j\right)=\begin{bmatrix} 1 \\ 1 \end{bmatrix}. $$
%
Figure \ref{fig:sim_1m_5} depicts the trajectory and the final position of the follower matching the analytic result. Figure \ref{fig:sim_1m_5_err}  verifies that the bearing error converges to zero exponentially fast.

\begin{figure}[h]
    \centering
    \begin{subfigure}[b]{0.48\linewidth}
        \centering
        \includegraphics[width=\textwidth]{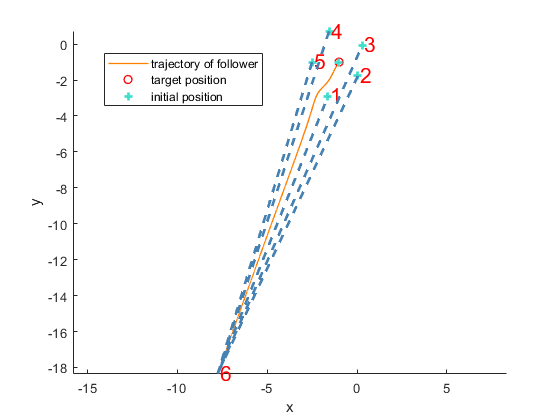}
        \caption{System trajectories.}\label{fig:sym_1m_5_tra}
    \end{subfigure}
    \begin{subfigure}[b]{0.48\linewidth}
        \centering
        \includegraphics[width=\textwidth]{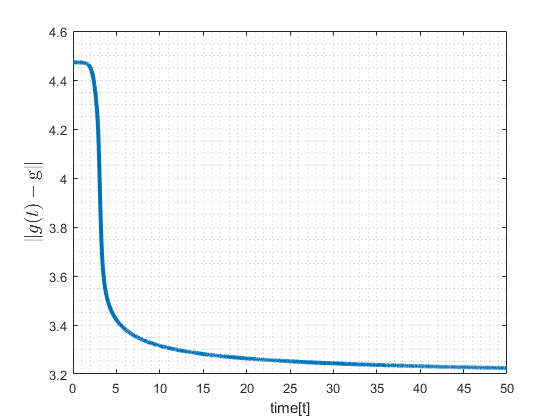}
        \caption{Bearing error along the system trajectories.}\label{fig:sym_1m_5_err}
    \end{subfigure}
    \caption{1-to-many system initialized near the unstable equilibrium.}
    \label{fig:1DOF_sym}
\end{figure}

We now consider the same setup but with initial conditions that are symmetric with respect to the origin. The leaders are placed at $\bar p= -[\mathrm p_1^T,\mathrm p_2^T,\mathrm p_3^T,\mathrm p_4^T,\mathrm p_5^T]^T$. From Lemma \ref{lem:unstable_1dof}, the system has a unique unstable equilibria at $\mathrm p_6=[1\ 1]^T$. In simulation, the initial position is taken near the unstable equilibria at $p_6(0)=[1.01\ 1.01]^T$.
As depicted in Fig.\ref{fig:1DOF_sym}, the follower diverges from the unstable equilibria, and the bearing error does not converge to zero.

\subsection{Performance Improvement of Ordered LFF Graphs}\label{subsec.convergence}

In this example, we demonstrate how ordered LFF graphs lead to faster convergence to the desired formation compared to the LFF graphs used in \cite{trinh2018bearing}.
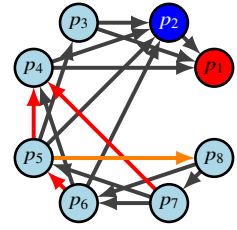
\begin{figure}
    \centering
    \begin{tikzpicture}[scale=0.6]
    \Vertex[size=.5,label=$p_1$,color=red,x=2,y=1]{v1}
    \Vertex[size=.5,label=$\color{white}{p_2}$,color=blue,x=1,y=2]{v2}
    \Vertex[size=.5,label=$p_3$,x=-1,y=2]{v3}
    \Vertex[size=.5,label=$p_4$,x=-2,y=1]{v4}
    \Vertex[size=.5,label=$p_5$,x=-2,y=-1]{v5}
    \Vertex[size=.5,label=$p_6$,x=-1,y=-2]{v6}
    \Vertex[size=.5,label=$p_7$,x=1,y=-2]{v7}
    \Vertex[size=.5,label=$p_8$,x=2,y=-1]{v8}
    \Edge[Direct](v2)(v1)
    \Edge[Direct](v3)(v1)
    \Edge[Direct](v3)(v2)
    \Edge[Direct](v4)(v1)
    \Edge[Direct](v4)(v2)
    \Edge[Direct](v5)(v2)
    \Edge[Direct](v5)(v3)
    \Edge[Direct](v6)(v2)
    \Edge[Direct](v6)(v4)
    \Edge[Direct](v7)(v5)
    \Edge[Direct](v7)(v6)
    \Edge[Direct](v8)(v6)
    \Edge[Direct](v8)(v7)
    \Edge[Direct,color=red](v5)(v4)
    \Edge[Direct,color=red](v6)(v5)
    \Edge[Direct,color=red](v7)(v4)
    \Edge[Direct,color=orange](v5)(v8)
    \end{tikzpicture}
    \caption{The sensing graphs for the examples in Sections \ref{subsec.convergence} and \ref{subsec.extension}.  The graph with only the black edges is LFF.  The graph with the black and red edges is an ordered LFF.  Finally, the graph with black, red, and orange is an unordered LFF.}
    \label{fig:Ex_graph}
\end{figure}
We consider the LFF graph in Fig.\ref{fig:Ex_graph} with black edges, and an ordered LFF graph obtained by adding the red edges.  The leader node is denoted in red ($p_1$) and the first follower in blue ($p_2$).


\begin{figure}[b]
    \begin{subfigure}[b]{.48\linewidth}
    \centering
        \includegraphics[width=\linewidth]{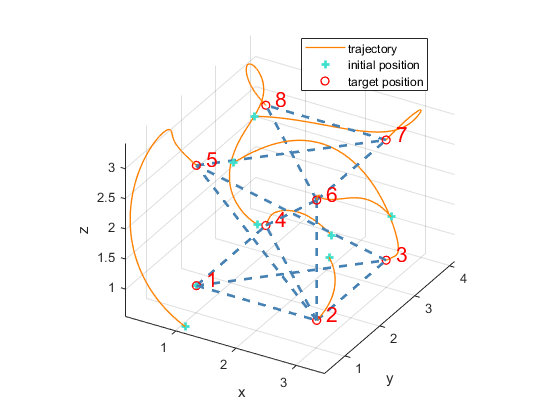}
        \caption{BOFC with LFF graph.}
        \label{fig:LFF_8a3d_tra}
    \end{subfigure}
     \begin{subfigure}[b]{.48\linewidth}
    \centering
        \includegraphics[width=\linewidth]{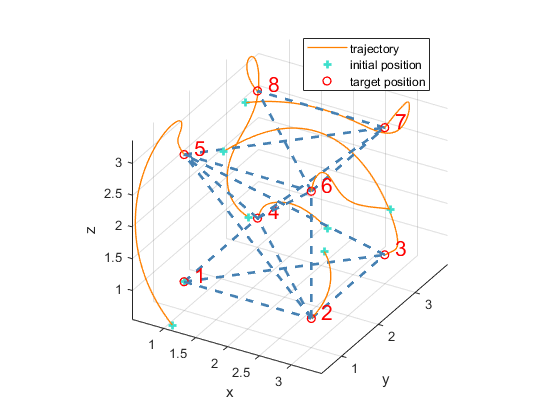}
        \caption{BOFC with ordered LFF graph. }
        \label{fig:OLFF_8a3d_tra}
    \end{subfigure}
    \caption{Trajectories for (a) the LFF graph, and (b) the ordered LFF graph.}
    \label{fig:LFF_tra}
\end{figure}

Figure \ref{fig:LFF_tra} shows the trajectories of the BOFC for the LFF (Fig. \ref{fig:LFF_8a3d_tra}) and ordered LFF (Fig. \ref{fig:OLFF_8a3d_tra}) for a target formation embedded in $\mathbb{R}^3$.  
The bearing error along the trajectories for each case is displayed in Fig.\ref{fig:OLFF_err}. The additional edges in the ordered LFF structure lead to a faster convergence rate to the target formation. 

\begin{figure}[h]
     \centering
        \includegraphics[width=.5\linewidth]{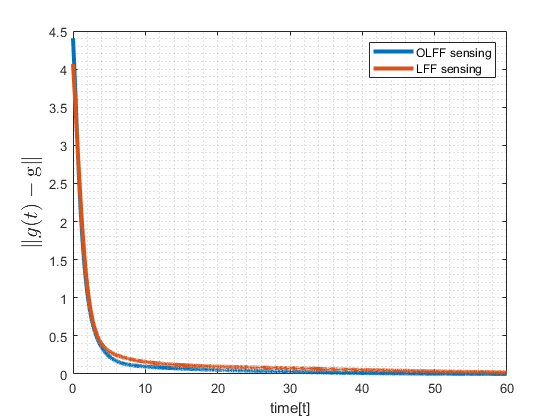}
    \caption{Caparison of bearing error along the system trajectories between the LFF and ordered LFF graphs.}
    \label{fig:OLFF_err}
\end{figure}

\subsection{Further Extensions to the LFF Graphs}\label{subsec.extension}


Finally, we demonstrate that there are additional graph structures that may still solve the BOFC problem.  We consider again the graph in Figure \ref{fig:Ex_graph} with the black, red, and orange edges.  Note that the orange edges are not forward edges, and therefore the graph is neither an LFF or ordered LFF graph.  On the other hand, there is an LFF subgraph in this structure.  Figure \ref{fig:DLFF_8a2d} shows the system trajectories and bearing error of the BOFC using this sensing graph.  The fact that the system converges to the correct target formation suggests there are additional structures that demand further examination.  Note that for this example, the resulting dynamics do not have a cascade structure so the methods used in this work can not apply to study its behavior.  Exploring these other structures is a subject of future work.

\begin{figure}[h]
    \centering
    \begin{subfigure}[b]{.48\linewidth}
    \centering
        \includegraphics[width=\linewidth]{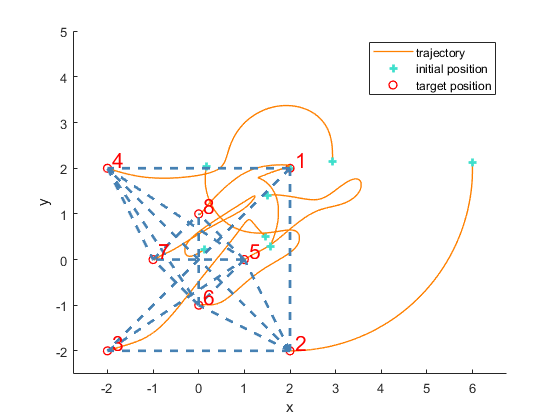}
        \caption{System trajectory.}
        \label{fig:DLFF_8a2d_tra}
    \end{subfigure}
    \begin{subfigure}[b]{.48\linewidth}
    \centering
        \includegraphics[width=\linewidth]{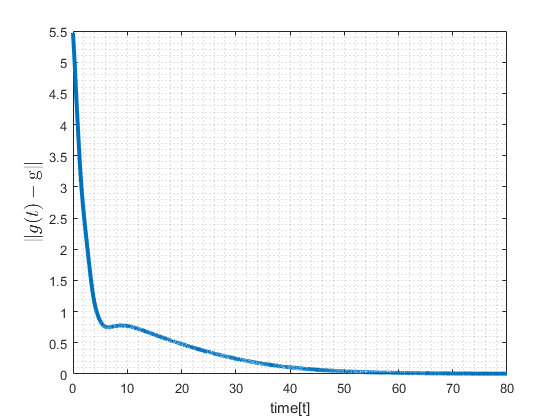}
        \caption{Bearing error along the system trajectories.}
        \label{fig:DLFF_8a2d_err}
    \end{subfigure}
    \caption{Trajectories for the BOFC strategy with a graph that is neither LFF or an ordered LFF graph.}
    \label{fig:DLFF_8a2d}
\end{figure}
    



\section{Conclusion}\label{sec.conclusion}

This work extends the applicability of bearing-only formation control to a broader class of directed sensing graphs by augmenting traditional LFF structures. We demonstrated how forward directed edges can be added to enhance performance while preserving stability, and our simulation results highlight the improved convergence and design flexibility of the proposed methods. Future research will focus on exploring more general directed topologies, incorporating robustness to sensing uncertainties, and addressing real-world constraints, such as communication delays and dynamic network changes, to bridge the gap between theoretical advancements and practical implementations.

\bibliographystyle{IEEEtran}
\bibliography{ref}
 
\end{document}